\documentclass[12pt]{elsarticle}
\usepackage{amssymb}
\usepackage{amsthm}
\usepackage{enumerate}
\usepackage{amsmath,amssymb,amsfonts,amsthm,fancyhdr}
\usepackage[colorlinks,linkcolor=black,anchorcolor=black,citecolor=black,CJKbookmarks=True]{hyperref}

\newdefinition{rem}{Remark}[section]
\newdefinition{theorem}{Theorem}[section]
\newdefinition{corollary}{Corollary}[section]
\newdefinition{definition}{Definition}[section]
\newdefinition{lemma}{Lemma}[section]
\newdefinition{prop}{Proposition}[section]
\numberwithin{equation}{section}
\def\bc{\begin{center}}
\def\ec{\end{center}}
\def\be{\begin{equation}}
\def\ee{\end{equation}}

\def\N{\mathbb N}

\newcommand\hdim{\dim_{\mathrm H}}

\begin{document}
\begin{frontmatter}
\title{Increasing rate of weighted product of partial quotients in continued fractions}
\author[a]{Ayreena~Bakhtawar}\ead{a.bakhtawar@unsw.edu.au}

\author[b]{Jing Feng\corref{cor1}}\ead{jingfeng0425@gmail.com}
 \address[a]{School of Mathematics and Statistics, University of New South Wales, Sydney NSW 2052, Australia.}
 \address[b]{School of Mathematics and Statistics, Huazhong University of Science and Technology, Wuhan, 430074 PR China and
LAMA UMR 8050, CNRS, Universit\'e  Paris-Est Cr\'eteil, 61 Avenue du G\'en\'eral de Gaulle, 94010 Cr\'eteil  Cedex, France }
 \cortext[cor1]{Corresponding author.}

\begin{abstract}\par
 Let $[a_1(x),a_2(x),\cdots,a_n(x),\cdots]$ be the continued fraction expansion of $x\in[0,1)$. In this paper, we study the increasing rate of the weighted product $a^{t_0}_n(x)a^{t_1}_{n+1}(x)\cdots a^{t_m}_{n+m}(x)$ ,where $t_i\in \mathbb{R}_+\ (0\leq i \leq m)$ are weights. More precisely, let $\varphi:\mathbb{N}\to\mathbb{R}_+$ be a function with $\varphi(n)/n\to \infty$ as $n\to \infty$. For any $(t_0,\cdots,t_m)\in \mathbb{R}^{m+1}_+$ with $t_i\geq 0$ and at least one $t_i\neq0 \
(0\leq i\leq m)$, the Hausdorff dimension of the set
$$\underline{E}(\{t_i\}_{i=0}^m,\varphi)=\left\{x\in[0,1):\liminf\limits_{n\to \infty}\dfrac{\log \left(a^{t_0}_n(x)a^{t_1}_{n+1}(x)\cdots a^{t_m}_{n+m}(x)\right)}{\varphi(n)}=1\right\}$$
is obtained. Under the condition that $(t_0,\cdots,t_m)\in \mathbb{R}^{m+1}_+$ with $0<t_0\leq t_1\leq \cdots \leq t_m$, we also obtain the Hausdorff dimension of the set
\begin{equation*}
\overline{E}(\{t_i\}_{i=0}^m,\varphi)=\left\{x\in[0,1):\limsup\limits_{n\to \infty}\dfrac{\log \left(a^{t_0}_n(x)a^{t_1}_{n+1}(x)\cdots a^{t_m}_{n+m}(x)\right)}{\varphi(n)}=1\right\}.
\end{equation*}

\end{abstract}
\begin{keyword}
 Continued fractions, Hausdorff dimension, Product of partial quotients
\MSC[2010] Primary 11K50, Secondary 37E05, 28A80
\end{keyword}
\end{frontmatter}

\section{Introduction}
Each irrational number $x\in[0,1)$ admits a unique continued fraction expansion of the from
\begin{equation}\label{cf11}
x=\frac{1}{a_{1}(x)+\displaystyle{\frac{1}{a_{2}(x)+\displaystyle{\frac{1}{
a_{3}(x)+{\ddots }}}}}},\end{equation}
where
for each $n\geq 1$, the positive integers $a_{n}(x)$ are known as the partial quotients of $x.$ The partial quotients can be generated by using
the Gauss transformation $T:[0,1)\rightarrow [0,1)$ defined as
\begin{equation} \label{Gmp}
T(0):=0 \ {\rm} \ {\rm and} \ T(x)=\frac{1}{x} \ ({\rm mod} \ 1), \text{ for } x\in (0,1).
\end{equation}
In fact, let $a_{1}(x)=\big\lfloor \frac{1}{x}\big\rfloor $ (where $\lfloor \cdot \rfloor$ denotes the integral part of real number). Then $a_{n}(x)=\big\lfloor \frac{1}{T^{n-1}(x)}\big\rfloor$ for $n\geq 2$. Sometimes \eqref{cf11} is written as
$x= [a_{1}(x),a_{2}(x),a_{3}(x),\ldots ].$
Further, the $n$-th convergent $p_n(x)/q_n(x)$ of $x$ is defined by $p_n(x)/q_n(x)=[a_{1}(x),a_{2}(x),\ldots, a_n(x)].$

The metrical aspect of the theory of continued fractions has been very well studied due to its close connections with Diophantine approximation.
%The main connection is that the convergents of a real number $x$ are good rational approximates for $x.$
For example, for any $\tau>0$ the famous Jarn\'{i}k-Besicovitch set
\begin{equation*}
{J}_\tau:=\left\{ x\in \lbrack 0,1):  \left|x-\frac pq\right|  <\frac{1}{q^{\tau+2}}    \ \
\mathrm{for\ infinitely\ many\ }(p,q)\in \mathbb{Z} \times \mathbb{N}\right\},
\end{equation*}
can be described by using continued fractions. In fact,
\begin{equation}\label{Jset}
{J}_\tau=\left\{ x\in [ 0,1):a_{n+1}(x)\geq q^{\tau}_{n}(x)\ \
\mathrm{for\ infinitely\ many\ }n\in \mathbb{N}\right\}.
\end{equation}
For further details about this connection we refer to \cite{Go_41}.
%Thus a real number $x$ is $\tau$-approximable if the partial quotients in its continued fraction expansion are growing fast.
Thus the growth rate of the partial quotients reveals how well a real
number can be approximated by rationals.

The well-known Borel-Bernstein Theorem \cite{Be_12,Bo_12} states that for Lebesgue almost all $x\in[0,1),$
$a_{n}(x)\geq\varphi(n)$ holds for finitely many $n^{\prime}s$ or infinitely many $n^{\prime}s$ according to the convergence or divergence of  the series $\sum_{n=1}^{\infty}{1}/{\varphi(n)}$ respectively. However, for rapidly growing function ${\varphi},$ the Borel-Bernstein Theorem does not give any conclusive information other than Lebesgue measure zero. To distinguish the sizes of zero Lebesgue measure sets, Hausdorff dimension is considered as an appropriate conception and has gained much importance in the metrical theory of continued fractions.
%The first dimensional result was due to Jarn\'{i}k who proved that the set of real numbers whose partials quotients are bounded has full Hausdorff dimension.
Jarn\'{i}k \cite{Ja_32} proved that the set of real numbers with bounded partials quotients has full Hausdorff dimension.
Later on, Good \cite{Go_41} showed that the Hausdorff dimension of the set of numbers whose partial quotients tend to infinity is one half.

After that, a lot of
work has been done in the direction of improving Borel-Bernstein Theorem, for example, the Hausdorff dimension of sets when partial quotients $a_n(x)$ obeys different conditions
has been obtained in \cite{FaLiWaWu_13,FaLiWaWu_09,FaMaSo_21,Luczak,LiRa_016,LiRa_16,WaWu_008}.
%The Hausdorff dimension of the set of points $x\in[0,1)$ such that
%$a_{n}(x)\geq\varphi(n)$ holds for infinitely many $n$ has been determined by Wang-Wu \cite{}.

Motivation for studying the growth rate of the products of consecutive partial quotients aroses from the works of Davenport-Schmidt \cite{DaSc_70} and Kleinbock-Wadleigh \cite{KlWa_18} where they considered improvements to Dirichlet's theorem.
Let $\psi :[t_{0},\infty
)\rightarrow \mathbb{R}_{+}$ be a monotonically decreasing function, where
$t_{0}\geq 1$ is fixed. Denote by $D(\psi )$ the set of all real numbers $
x$ for which the system
        \begin{equation*}
|qx-p|\leq \psi (t)\ \rm{and}  \ |q|<t
\end{equation*}
has a nontrivial integer solution for all large enough $t$. A real number $
x\in D(\psi )$ (resp. $x\in D(\psi )^{c}$) will be referred to as \emph{$
\psi$-Dirichlet improvable} (resp. \emph{$\psi$-Dirichlet non-improvable})
number.

The starting point for the work of Davenport-Schmidt \cite{DaSc_70} and Kleinbock-Wadleigh {\protect\cite[Lemma 2.2]{KlWa_18}} is an observation that Dirichlet improvability is equivalent to a condition on the growth rate of the products of two consecutive partial quotients. Precisely, they observed that
\begin{align*}
x\in D(\psi) &\Longleftrightarrow |q_{n-1}x-p_{n-1}|\le \psi(q_n),\text{ for all } n \text{ large } \\
&\Longleftrightarrow [a_{n+1}, a_{n+2},\cdots]\cdot [a_n, a_{n-1},\cdots, a_1]\ge ((q_n\psi(q_n))^{-1}-1), \text{ for all } n \text{ large. }
\end{align*}
Then
\begin{multline*}
\Big\{x\in [0,1)\colon a_n(x)a_{n+1}(x)\ge ((q_n(x)\psi(q_n(x)))^{-1}-1)^{-1} \ {\text{for i.m.}}\ n\in \N\Big\}\subset D^{\mathsf{c}}(\psi)\\
\subset \Big\{x\in [0,1)\colon a_n(x)a_{n+1}(x)\ge 4^{-1}((q_n(x)\psi(q_n(x)))^{-1}-1)^{-1}\ {\text{for i.m.}}\ n\in \N\Big\},
\end{multline*}
where i.m. stands for infinitely many.

In other words, a real number $x\in [0, 1)\setminus\mathbb{Q}$ is $\psi$-Dirichlet improvable if and only if the products of consecutive partial quotients of $x$ do not grow quickly. We refer the reader to \cite{Ba_20,BaHuKlWa_22,FeXu_21,HuaWu_19,HuWuXu_19} for more metrical results related with the set of Dirichlet non-improvable numbers.

%\begin{thm}[Zhang, \cite{Zh_20}]
%%\label{ASLilithm}
%Let $\Phi :\mathbb{N}\rightarrow (1,\infty)$ be a function with $\lim_{n\to \infty} \Phi(n)=~\infty.$ Then
%
%$$ \dim_H N\mathcal D(\Phi) =\frac1{b+1}  \text{ where } \log b = \limsup_{n\to\infty} \frac{\log\log\Phi(n)}{n}.$$
%\end{thm}

%%%%%%%%%%%%%
%When $\Phi$ is a function of the $n$th convergents ($q_n(x)$)  then the Hausdorff dimension of the set $\mathcal{F}(\Phi )$ has been  established by the authors in \cite{BaBoHu18},  and the Hausdorff measure theoretic results for the set  $\mathcal{E}_{2}(\Phi)$   have  been established by Hussain-Kleinbock-Wadleigh-Wang in \cite{HKWaW17}. The Hausdorff dimension of level sets within this setup are investigated very recently by Huang-Wu  \cite{LiHu19}.

%%%%%%%%%%%%%%%%%%%

%\subsection{Main result}
%Let $\varphi$ be a positive function defined on $\mathbb{N}$ with $\varphi(n)\rightarrow\infty$ as $n\rightarrow \infty$. For any $m\geq0$ and $0\leq i \leq m$, $t_i\in \mathbb{N}\bigcup\{0\}$, define

As a consequence of  Borel-Bernstein Theorem, for almost all $x\in[0,1)$ there exists a subsequence of partial quotients tending to infinity with a linear speed. In other words, for Lebesgue almost every $x\in[0,1)$
\begin{equation*}
\limsup_{n\to\infty}\frac{\log a_{n}(x)}{\log n}=1.
\end{equation*}

Taking inspirations from the study of the growth rate of the products of consecutive partial
quotients for the real numbers, in this paper we consider the growth rate of the products of the consecutive weighted partial quotients. More precisely, by \cite[Theorem 1.4]{BaHuKlWa_22}, we have for Lebesgue almost all $x\in[0,1)$
\begin{equation}\label{ES2}
\limsup_{n\to\infty}\frac{\log a^{t_0}_{n}(x)a^{t_1}_{n+1}(x)\cdots a^{t_{m}}_{n+m}(x)}{\log n^{t_{\max}}}=1,
\end{equation}
where $t_{\max}=\max\{t_i:0\leq i\leq m\}$.
This paper is concerned with Hausdorff dimension of
some exceptional sets of \eqref{ES2}. Let $\varphi: \mathbb{N}\to \mathbb{R}_+$ be a function satisfying  $\varphi(n)/n\to \infty$ as $n\to\infty $ and let $t_i\in \mathbb{R}_+\ (0\leq i \leq m)$. Define the sets
\begin{equation*}
\overline{E}(\{t_i\}_{i=0}^m,\varphi)=\left\{x\in[0,1):\limsup\limits_{n\to \infty}\dfrac{\log \left(a^{t_0}_n(x)a^{t_1}_{n+1}(x)\cdots a^{t_m}_{n+m}(x)\right)}{\varphi(n)}=1\right\},
\end{equation*}
and
\begin{equation*}
\underline{E}(\{t_i\}_{i=0}^m,\varphi)=\left\{x\in[0,1):\liminf\limits_{n\to \infty}\dfrac{\log \left(a^{t_0}_n(x)a^{t_1}_{n+1}(x)\cdots a^{t_m}_{n+m}(x)\right)}{\varphi(n)}=1\right\}.
\end{equation*}

The study of the level sets about the growth rate of $\{a_n(x)a_{n+1}(x) : n \geq
1\}$ relative to that of $\{q_n(x) : n \geq 1\}$ was discussed in \cite{HuaWu_19}.
Let $m\geq2$ be an integer and $\Psi: \mathbb{N}\to \mathbb{R}_{+}$ be a positive function. The Lebesgue measure and the Hausdorff dimension of the set
 \begin{equation}\label{em}
\left\{x\in[0, 1): {a_{n}(x)\cdots a_{n+m}(x)}\geq \Psi(n) \ \text{for infinitely many} \ n\in \N \right\},
\end{equation}
have been comprehensively determined by Huang-Wu-Xu \cite{HuWuXu_19}.  Very recently
the results of \cite{HuWuXu_19} were generalized by Bakhtawar-Hussain-Kleinbock-Wang \cite{BaHuKlWa_22}
to a weighted generalization of the set \eqref{em}. For more details we refer the reader to
\cite{BaHuKlWa_22,HuWuXu_19}.

%Throughout this paper, we denote $\mathbb{N}=\{1,2,3,\cdots\}.$
%The set $\overline{E}(\{t_i\}_{i=0}^m,\varphi)$ can be viewed as the weighted generalization of exceptional sets of Borel Bernstein theorem in the sense that here we are considering the growth rate of weighted product of consecutive partial quotients whereas BB focuses on the growth rate of one partial quotients.
%Besides the set we are also interested in the Hausdorff dimension of the set
Our main results are as follows.
\begin{theorem}\label{foralln} Let $\varphi: \mathbb{N}\to \mathbb{R}_+$ be a function satisfying $\varphi(n)/n\to \infty$ as $n\to\infty $. Write
\begin{equation*} \log B=\limsup\limits_{n\rightarrow \infty} \frac{ \log \varphi(n)}{n}.
\end{equation*}
Assume $B\in[1,\infty]$. Then for any $(t_0,\cdots,t_m)\in \mathbb{R}^{m+1}_+$ with $t_i\geq 0$ and at least one $t_i\neq0 \
(0\leq i\leq m)$,
\begin{equation*}
\hdim \underline{E}(\{t_i\}_{i=0}^m,\varphi)=\frac{1}{1+B}.
\end{equation*}
\end{theorem}

\begin{theorem}\label{forinfinitelymanyn}
Let $\varphi: \mathbb{N}\to \mathbb{R}_+$ be a function satisfying $\varphi(n)/n\to \infty$ as $n\to\infty $.
Write
\begin{equation*} \log b=\liminf\limits_{n\rightarrow \infty} \frac{ \log \varphi(n)}{n}.
\end{equation*}
Assume $b\in[1,\infty]$. Then for any $(t_0,t_1,\cdots,t_m)\in \mathbb{R}_+^{m+1}$ with $0<t_0\leq t_1\leq \cdots \leq t_m$, we have
\begin{equation*}
\hdim \overline{E}(\{t_i\}_{i=0}^m,\varphi)=\frac{1}{1+b}.
\end{equation*}

\end{theorem}

\begin{rem}
  Note that in Theorem \ref{forinfinitelymanyn} we are only able to treat the case when the sequence
$\{t_{i}\}_{i=0}^{m}$ is nondecreasing. We would like to drop this monotonic condition. Indeed, our method for the upper bound is true for all sequences $\{t_{i}\}_{i=0}^{m}$. However, when dealing with the lower bound, the sequence $\{c_n\}_{n\geq1}$ we construct (see the proof for details) might not be bounded away from 0 once we drop the monotonic condition, which is important in constructing a suitable subset of $\overline{E}(\{t_i\}_{i=0}^m,\varphi)$.
\end{rem}

\section{Preliminaries}
In this section, we fix some notations and recall some known results in theory of continued fraction expansions.

For an irrational number $x\in[0,1)$, recall  $a_n(x) $ is the $n$-th partial quotient of $x$ in its continued fraction expansion. The sequences $\{p_n(x)\}_{n\geq1},$ $\{q_n(x)\}_{n\geq1}$ are the numerator and denominator of the $n$-th convergent of $x$. It is well-known that $\{p_n(x)\}_{n\geq1}$ and $\{q_n(x)\}_{n\geq1}$ can be obtain by the following recursive relations (see \cite{Kh_63}):

\begin{equation}\label{sequence-qn}
\begin{cases}
&p_n(x)=a_n(x)p_{n-1}(x)+p_{n-2}(x),\ \text{for any $n \geq $ 1 };\\
&q_n(x)=a_n(x)q_{n-1}(x)+q_{n-2}(x),\ \text{for any $n \geq $ 1 },
\end{cases}
\end{equation}
with the conventions $p_{-1}=1,\  q_{-1}=0, \  p_0=0$ and $q_0=1$.

For any $n$-tuple $(a_1,\cdots,a_n)\in \mathbb{N}^n$ with $n\geq 1$, we call
$$I_n(a_1,\cdots,a_n)=\big\{x\in[0,1)\colon a_1(x)=a_1,\ a_2(x)=a_2, \cdots ,a_n(x)=a_n\big\},$$
a cylinder of order $n$.

Note that $p_n(x)$ and $q_n(x)$ are determined by the first $n$ partial quotients of $x$. So all points in $I_n(a_1,\cdots,a_n)$ determine the same $p_n(x)$ and $q_n(x)$. Hence for simplicity, if there is no confusion, we write $a_n$, $p_n$ and $q_n$ to denote $a_n(x)$, $p_n(x)$ and $q_n(x)$ for $x\in I_n(a_1,\cdots,a_n)$ respectively.

The following lemma is a collection of basic facts on continued fractions which can be found in the book of Khintchine \cite{Kh_63}.

\begin{lemma}
For any $(a_1,\cdots,a_n)\in \mathbb{N}^n$, let $q_n$ and $p_n$ be given recursively by \eqref{sequence-qn}. Then

(1)\begin{equation*}
I_n(a_1,\cdots,a_n)=\begin{cases}
&\Big[\dfrac{p_n}{q_n},\dfrac{p_n+p_{n-1}}{q_n+q_{n-1}}\Big),  \text{ if } n \text{ is even },\\
&\Big[\dfrac{p_n+p_{n-1}}{q_n+q_{n-1}},\dfrac{p_n}{q_n}\Big),  \text{ if } n \text{ is odd };
\end{cases}
\end{equation*}

(2) $q_n\geq 2^{\frac{n-1}{2}},\ \prod_{k=1}^na_k\leq q_n \leq  2^n\prod_{k=1}^na_k;$

(3)
\begin{equation*}
\frac{1}{3a_{n+1}q_{n}^{2}}\,<\,\Big|x-\frac{p_{n}}{q_{n}}\Big|=\frac{1}{
q_n(q_{n+1}+T^{n+1}(x)q_n)}<\,\frac{1}{a_{n+1}q_{n}^{2}},
\end{equation*}
and for any $n\geq1$ the derivative of $T^{n}$ is given by

\begin{equation*}
(T^{n})^{\prime }(x)=\frac{(-1)^{n}}{(xq_{n-1}-p_{n-1})^{2}}.
\end{equation*}
\end{lemma}

The next theorem, known as Legendre's Theorem, connects
1-dimensional Diophantine approximation with continued fractions.
\begin{theorem}[Legendre]\label{leg}
Let $\frac{p}{q}$ be an irreducible rational number. Then
\begin{equation*}  \label{Legendre}
\Big|x-\frac pq\Big|<\frac1{2q^2}\Longrightarrow \frac pq=\frac{p_n(x)}{
q_n(x)},\quad \mathrm{for\ some \ } n\geq 1.
\end{equation*}
\end{theorem}
According to Legendre's theorem if an irrational $x$ is ``well" approximated
by a rational $p/q$, then this rational must be a convergent of $x$. So, the continued fraction expansions is a quick and efficient tool for finding good rational approximations to real numbers.
%Thus in order to find good rational approximates to an irrational number we
%only need to focus on its convergents. Note that, from $\rm{(3)}$
%of Lemma \ref{9},  a real number $x$ is well approximated by its
%convergent $\frac{p_{n}}{q_{n}}$ if its $(n+1)$th partial quotient $a_{n+1}$ is sufficiently large.
For more basic properties of continued fraction expansions, one can refer to \cite{Kh_63}. We also give some auxiliary results on the Hausdorff dimension theory of continued fractions that will be used later.

 \begin{lemma}[\cite{FaLiWaWu_13}]\label{anyset}
Let $\{s_n\}_{n\geq1}$ be a sequence of positive integers tending to infinity, then for any positive integer number $N\geq2$,
\begin{equation*}
\begin{aligned}
&\hdim \big\{x\in[0,1): s_n\leq a_n(x)< Ns_n,\  \text{for all}\ n \geq 1\big\}\\
=&\liminf_{n\to\infty}\frac{\log (s_1s_2\cdots s_n)}{2\log (s_1s_2\cdots s_n)+\log s_{n+1}}\\
=&\frac{1}{2+\limsup\limits_{n\to\infty}\frac{\log s_{n+1}}{\log (s_1s_2\cdots s_n)}}.
\end{aligned}
\end{equation*}
\end{lemma}

\begin{lemma} [\cite{FeWULiTs_97,Luczak}]\label{LUZARK RESULT}
 For any  $a,c>1$,
\begin{equation*}
\begin{aligned}
&\hdim \left\{x\in[0,1):a_n(x)\geq c^{a^n},\ \text{for all} \ n\geq1\right\}\\
=&\hdim \left\{x\in[0,1):a_n(x)\geq c^{a^n},\ \text{for infinitely many}\  n\in\mathbb{N}\right\}\\
=&\frac{1}{1+a}.
\end{aligned}
\end{equation*}
\end{lemma}

Applying Lemma \ref{LUZARK RESULT}, we obtain the follwing corollary which will be useful for the upper bound estimation on $\hdim \underline{E}(\{t_i\}_{i=0}^m,\varphi).$

\begin{corollary}\label{Luzarks cor}
For any $a,c>1$ and $(t_0,\cdots,t_m)\in \mathbb{R}^{m+1}_+$ with $t_i\geq 0$ and at least one $t_i\neq0 \
(0\leq i\leq m)$,
\begin{equation*}
\begin{aligned}
&\hdim \left\{x\in[0,1):a^{t_0}_n(x)a^{t_1}_{n+1}(x)\cdots a^{t_m}_{n+m}(x)\geq c^{a^n},\ \text{for all} \ n\geq1\right\}\\
=&\hdim \left\{x\in[0,1):a^{t_0}_n(x)a^{t_1}_{n+1}(x)\cdots a^{t_m}_{n+m}(x)\geq c^{a^n},\ \text{for infinitely many}\  n\geq1\right\}\\
=&\frac{1}{1+a}.
\end{aligned}
\end{equation*}
\end{corollary}
\begin{proof}Denote $k=\min\{0\leq i\leq m:\ t_{i}\neq0\}$. It is clear that for some $t_{j}\neq0 \ (0\leq j\leq m),$
\begin{equation*}
\begin{aligned}
&\left\{x\in[0,1):a^{t_k}_{n+k}(x)\geq c^{a^n},\ \text{for all} \ n\geq1\right\}
\\
\subset&\left\{x\in[0,1):a^{t_0}_n(x)a^{t_1}_{n+1}(x)\cdots a^{t_m}_{n+m}(x)\geq c^{a^n},\ \text{ for all} \ n\geq1\right\}\\
\subset& \left\{x\in[0,1):a^{t_0}_n(x)a^{t_1}_{n+1}(x)\cdots a^{t_m}_{n+m}(x)\geq c^{a^n},\ \text{ for infinitely many} \ n\geq1\right\}\\
\subset&\left\{x\in[0,1):a^{t_{j}}_{n+j}(x)\geq c^{\frac{a^n}{m+1}},\ \text{ for infinitely many}\  n\geq1\right\}.
\end{aligned}
\end{equation*}
From Lemma \ref{LUZARK RESULT}, we deduce that for some $t_{j}\neq0\ (0\leq j\leq m),$
\begin{equation*}
\begin{aligned}
&\hdim \left\{x\in[0,1):a^{t_{k}}_{n+k}(x)\geq c^{a^n},\ \text{ for  all} \ n\geq1\right\}\\
=&\hdim \left\{x\in[0,1):a^{t_{j}}_{n+j}(x)\geq c^{a^n},\ \text{for infinitely many}\  n\in\mathbb{N}\right\}\\
=&\frac{1}{1+a}.
\end{aligned}
\end{equation*}
Then the desired results are obtained.
\end{proof}

\section{Proof of theorem \ref{foralln}}

%\subsection{ Proof of Theorem \ref{foralln}}

Let $\varphi: \mathbb{N}\to \mathbb{R}_+$ be a positive function with $\varphi(n)\rightarrow\infty$ as $n\rightarrow \infty$. For any $(t_0,\cdots,t_m)\in \mathbb{R}^{m+1}_+$ with $t_i\geq 0$ and at least one $t_i\neq0 \
(0\leq i\leq m)$, we introduce the sets
\begin{align*}
ND(\varphi)=\left\{x\in[0,1):a^{t_0}_n(x)a^{t_1}_{n+1}(x)\cdots a^{t_m}_{n+m}(x)\geq \varphi(n),\text{ for all }\ n \geq 1\right\},
\end{align*}
and
\begin{align*}
{ND}^{\prime}(\varphi)=\left\{x\in[0,1):a^{t_0}_n(x)a^{t_1}_{n+1}(x)\cdots a^{t_m}_{n+m}(x)\geq \varphi(n),  \text{ for $n$ large enough }\right\}.
\end{align*}
In order to prove Theorem \ref{foralln}, we first give a complete characterization on the size of the sets $ND(\varphi)$ and ${ND}^{\prime}(\varphi)$
 in terms of Hausdorff dimension.

\begin{prop}\label{PQ set}For any  $(t_0,\cdots,t_m)\in \mathbb{R}^{m+1}_+$ with $t_i\geq 0$ and at least one $t_i\neq0~ (0\leq i\leq m)$,
\begin{equation*}
\hdim ND(\varphi)=\hdim {ND}^{\prime}(\varphi)=\frac{1}{1+A},\ \text{where}\ \log A=\limsup\limits_{n\rightarrow \infty} \frac{\log \log \varphi(n)}{n}.
\end{equation*}
\end{prop}

We remark that recently Zhang (\cite{Zh_20}) obtained the Hausdorff dimension results of $ND(\varphi)$ and ${ND}^{\prime}(\varphi)$ for the special case $t_{0}=t_{1}=1$, $t_{i}=0\ (i\geq2).$

\subsection{ Proof of  Proposition \ref{PQ set} }\
\\
To prove Proposition \ref{PQ set}, we start with the following lemma.
\begin{lemma} \label{a_ninfinity}
For any $(t_0,\cdots,t_m)\in \mathbb{R}^{m+1}_+$ with $t_i\geq 0$ and at least one $t_i\neq0~(0\leq i\leq m)$,
\begin{equation*}
\hdim \left\{x\in[0,1): a^{t_0}_n(x)a^{t_1}_{n+1}(x)\cdots a^{t_m}_{n+m}(x)\rightarrow \infty, \ \text{as}\ n\rightarrow \infty\right\}=\frac 1 2.
\end{equation*}
\end{lemma}
\begin{proof} Denote by $C(\infty)$ the set above and $k=\min\{0\leq i\leq m:\ t_{i}\neq0\}$. It is evident that
\begin{equation*}
\begin{aligned}
&\left\{x\in[0,1): a^{t_k}_{n+k}(x)\rightarrow \infty, \ \text{as}\ n\rightarrow \infty\right\}\\
\subset& \left\{x\in[0,1): a^{t_0}_n(x)a^{t_1}_{n+1}(x)\cdots a^{t_m}_{n+m}(x)\rightarrow \infty, \ \text{as}\ n\rightarrow \infty\right\}.
\end{aligned}
\end{equation*}
So, $\hdim C(\infty)\geq \frac 1 2.$

In the following, we give the upper bound for $\hdim C(\infty)$.

\textbf{Step \uppercase\expandafter{\romannumeral1}.} We find a cover for $C(\infty)$. For any $M>0$,
 \begin{equation*}
 \begin{aligned}
 C(\infty)&\subset \left\{x\in[0,1): a^{t_0}_n(x)a^{t_1}_{n+1}(x)\cdots a^{t_m}_{n+m}(x)\geq M, \ \text{for $n$ large enough}\right\}\\
 &=\bigcup_{N=1}^\infty \left\{x\in[0,1): a^{t_0}_n(x)a^{t_1}_{n+1}(x)\cdots a^{t_m}_{n+m}(x)\geq M, \ \text{for}\ n\geq N\right\}\\
 &:=\bigcup_{N=1}^\infty E_M(N).
 \end{aligned}
 \end{equation*}
It is clear that $\hdim C(\infty)\leq\hdim E_M(1)$, since $\hdim C(\infty)\leq\hdim \sup\limits_{N\geq1}E_M(N)$, and by \cite[Lemma 1]{Go_41}, we have $\hdim E_M(N)=\hdim E_M(1)$ for any $N\geq 1$. So it is sufficient to estimate the upper bound for $E_M(1)$.
For any $n\geq1$, set
$$D_n(M)=\left\{(a_1,\cdots,a_n)\in\mathbb{N}^n:\ a^{t_0}_ka^{t_1}_{k+1}\cdots a^{t_m}_{k+m}\geq M, \ \text{for}\ 1\leq k \leq n-m\right\}.$$
Hence,
\begin{equation}
\begin{aligned}\label{cover}
E_M(1)\subset &\bigcup_{(a_1,\cdots,a_{n})\in D_{n}(M)}I_n(a_1,\cdots,a_n).
\end{aligned}
\end{equation}
 \medskip
\textbf{Step \uppercase\expandafter{\romannumeral2}.} We construct a family of Bernoulli measures $\{\mu_t\}_{t>1}$ on $[0, 1)$. For each $t>1$ and any $(a_1,\cdots,a_n)\in \mathbb{N}^n,$ put
$$\mu_t(I_n(a_1,\cdots,a_n))=e^{-nP(t)-t\Sigma_{j=1}^n\log a_j},$$
where $e^{P(t)}=\sum_{k=1}^\infty k^{-t}.$ It is easy to see that
$$\sum_{a_{n+1}}\mu_t(I_n(a_1,\cdots,a_n,a_{n+1}))=\mu_t(I_n(a_1,\cdots,a_n))$$
and
$$\sum_{(a_1,\cdots,a_n)\in\mathbb{N}^n}\mu_t(I_n(a_1,\cdots,a_n))=1.$$
So the measures $\{\mu_t\}_{t>1}$ are well defined by Kolmogorov's consistency theorem.

Fix $s>\frac 1 2$ and set $t=s+\frac1 2>1$. Choose $M$ sufficiently large such that
\begin{equation}\label{bernoulli-estimation}
\left(s-\frac{1}{2}\right)\cdot\frac{\log M^{t^{-1}_{\max}}}{2m}\geq P(s+\frac{1}{2}),
\end{equation}
where $t_{\max}=\max\{t_i:0\leq i\leq m\}$.

 We claim that for any $(a_1,\cdots,a_n)\in D_n(M)$,
\begin{equation}\label{estimate for q_n}
q_n^{-2s}\leq\mu_{s+\frac 1 2}(I_n(a_1,\cdots,a_n)).
\end{equation}
More precisely, for any $(a_1,\cdots,a_n)\in D_n(M)$, by the fact that for $1\leq l\leq n-m$,
$$a^{t_0}_la^{t_1}_{l+1}\cdots a^{t_m}_{l+m}\geq M,$$
then for $1\leq l\leq n-m$,
$$a_la_{l+1}\cdots a_{l+m}\geq M^{t^{-1}_{\max}},$$
where $t_{\max}=\max\{t_i:0\leq i\leq m\}$. Then we have
\begin{equation*}
 e^{-2s\sum_{j=1}^n\log a_j}\leq e^{-(s+\frac 1 2)\sum_{j=1}^n\log a_j-(s-\frac 1 2) \lfloor \frac {n}{ m}\rfloor \log M^{t^{-1}_{\max}}}.
\end{equation*}
 Thus, by $q_n\geq \prod\limits_{i=1}^na_i$ and then \eqref{bernoulli-estimation}, we get
\begin{equation*}
q_n^{-2s}\leq e^{-2s\sum_{j=1}^n\log a_j}\leq e^{-(s+\frac 1 2)\sum_{j=1}^n\log a_j-nP(s+\frac 1 2)}
\end{equation*}
Therefore, by \eqref{cover} and \eqref{estimate for q_n},
\begin{equation*}
\begin{aligned}
\mathcal{H}^{s}(E_M(1))&\leq \liminf_{n\to \infty}\sum_{(a_1,\cdots,a_{n})\in D_{n}(M)}{\big|I_n(a_1,\cdots,a_{n})\big|}^s\\
&\leq \liminf_{n\to \infty}\sum_{(a_1,\cdots,a_{n})\in D_{n}(M)}\frac{1}{q_n^{2s}}\\
&\leq\liminf_{n\to \infty}\sum_{(a_1,\cdots,a_{n})\in D_{n}(M)}\mu_{s+\frac{1}{2}}(I_{n}(a_1,\cdots,a_{n}))=1.
\end{aligned}
\end{equation*}
 Hence $\hdim E_M(1)\leq s$, and then $\hdim C(\infty)\leq s$. Consequently, $\hdim C(\infty)\leq \frac 1 2$ by the arbitrariness of $s>\frac 1 2$.
This completes the proof of Lemma \ref{a_ninfinity}.
\end{proof}

Now we are ready to prove Proposition \ref{PQ set}.\\

\textbf{\noindent\text{Proof of Proposition \ref{PQ set}:}} We see that $\hdim {ND}^{\prime}(\varphi)=\hdim ND(\varphi).$ The proof is divided into two cases according to $A=1$ or $A>1$.

\emph{(1)} If $A=1,$ then for any $\epsilon>0$, $\varphi(n)\leq e^{ {(1+\epsilon)}^n}\ \text{holds for $n$ large enough},$ and we have
$$\big\{x\in [0,1): \ a_n(x)\geq e^{ {(1+\epsilon)}^n}\ \text{for $n$ large enough}\big\}\subset {ND}^{\prime}(\varphi).$$
By Lemma \ref{LUZARK RESULT}, we obtain
$$\hdim{ND}^{\prime}(\varphi)\geq \frac{1}{2}.$$

On the other hand,
$$ND(\varphi)\subset \left\{x\in[0,1): a^{t_0}_n(x)a^{t_1}_{n+1}(x)\cdots a^{t_m}_{n+m}(x)\rightarrow \infty, \ \text{as}\ n\rightarrow \infty\right\}.$$
Thus,
$$\hdim ND(\varphi)\leq \frac 1 2 .$$

\emph{(2)} If $A>1$, by the definition of limsup, for any $\epsilon>0$,
\begin{equation*}
\begin{cases}&\varphi(n)\geq e^{{(A-\epsilon)}^n },\; \text{for infinitely many} \ n \in\mathbb{N},\\
&\varphi(n)\leq e^{{(A+\epsilon)}^n },\; \text{for all sufficiently large}\;n.
\end{cases}
\end{equation*}
Therefore
$$\left\{x\in[0,1):a^{t_0}_n(x)a^{t_1}_{n+1}(x)\cdots a^{t_m}_{n+m}(x)\geq e^{ {(A+\epsilon)}^n}, \ \text{for any $n\geq1$}\right\}\subset {ND}^{\prime}(\varphi).$$
Applying Corollary \ref{Luzarks cor}, we obtain
\begin{equation*}\hdim{ND}^{\prime}(\varphi)\geq \frac{1}{1+A+\epsilon}.
\end{equation*}
By the arbritrary of $\epsilon$, we have $$\hdim{ND}^{\prime}(\varphi)\geq\frac{1}{1+A}.$$

On the other hand,

$$ND(\varphi)\subset\left\{x\in[0,1):a^{t_0}_n(x)a^{t_1}_{n+1}(x)\cdots a^{t_m}_{n+m}(x)\geq e^{ {(A-\epsilon)}^n}, \ \text{for infinitely many}\ n\in \mathbb{N}\right\}.$$
From Corollary \ref{Luzarks cor}, we obtain
\begin{equation*}\hdim ND(\varphi)\leq \frac{1}{1+A-\epsilon}.
\end{equation*}
Taking $\epsilon\to0$, we conclude
$$\hdim ND(\varphi)\leq \frac{1}{1+A}.$$
\qed \
\\

Let us give a proof of Theorem \ref{foralln}.

\subsection{ Proof of  Theorem  \ref{foralln} }\

\textbf{Upper bound:} For $x\in\underline{E}(\{t_i\}_{i=0}^m,\varphi)$, for any $\epsilon>0$, we have
$$a^{t_0}_n(x)a^{t_1}_{n+1}(x)\cdots a^{t_m}_{n+m}\geq e^{(1-\epsilon)\varphi(n)}, \ \text{for $n$ large enough}.$$
Then it follows from Proposition \ref{PQ set} that

\begin{equation*}
\hdim \underline{E}(\{t_i\}_{i=0}^m,\varphi)\leq\frac{1}{1+B}.
\end{equation*}

\textbf{Lower bound:} It is trivial for $B=\infty$, so we only need to consider the case $1\leq B<\infty$.

We construct a suitable Cantor subset of $\underline{E}(\{t_i\}_{i=0}^m,\varphi)$ in two steps.

\textbf{Step \uppercase\expandafter{\romannumeral1}.} Since $ \log B=\limsup\limits_{n\rightarrow \infty} \frac{ \log \varphi(n)}{n },$ for any $\epsilon>0$, we have $\varphi(n)\leq {(B+\frac \epsilon 2)}^n$ for $n$ large enough. Hence,
\begin{equation*}
\varphi(n){(B+\epsilon)}^{j-n}\leq {(B+\frac \epsilon 2)}^n{(B+\epsilon)}^{j-n}\to 0, \ \text{as} \ n\to \infty.
\end{equation*}
We define a sequence $\{L_j\}_{j\geq 1}$: For $j,k\geq1$, let
\begin{equation*}
 c_{j,k}=\begin{cases}&\exp({\varphi(k)}),  \ \ \ \ \ \ \ \ \ \ \ \ \ \ \  1\leq k \leq j;\\
 & \exp(\varphi(k){(B+\epsilon)}^{j-k}), \ k\geq j+1.
 \end{cases}
\end{equation*}
Define $L_j=\sup\limits_{k\geq 1}\{c_{j,k}\}$.
Clearly,
\begin{equation}\label{L j+1-L_j}
L_j\leq L_{j+1}\leq L_{j}^{B+\epsilon}\ \text{and}\ L_j\geq e^{\varphi(j)}\ \text{for any}\ j\geq 1.
\end{equation}
By the first part of \eqref{L j+1-L_j},
\begin{equation*}
\log L_{j+1}-\log L_j\leq (B+\epsilon-1)\log L_j.
\end{equation*}
Hence
\begin{equation}\label{for limsup set}
\log L_{n+1}-\log L_1\leq (B+\epsilon-1)\sum\limits_{j=1}^n\log L_j.
\end{equation}
We claim that
\begin{equation}\label{L j+1-L_j-2}
\liminf_{n\to\infty}\dfrac{\log L_n}{\varphi(n)}=1.
\end{equation}
In fact, on the one hand, in view of the second part of \eqref{L j+1-L_j}, we see at once that
$$\liminf_{n\to\infty}\dfrac{\log L_n}{\varphi(n)}\geq1.$$

For the opposite inequality, let $t_j:=\min\{k\geq1: c_{j,k}=L_j\}$. Notice that for many consecutive $j^{\prime}$s, the number $t_j$ will be the same. More precisely, if $t_j<j$, $t_{t_j}=t_{t_j+1}=\cdots=t_j;$ if  $t_j\geq j$, $t_{j}=t_{j+1}=\cdots=t_{t_j}.$ Let $\{l_i\}$ be the sequence of all ${t_{t_j}}^{\prime}$s in the strictly increasing order. Then we obtain $L_{l_i}=\exp{\varphi(l_i)}$ and thus
$$\liminf_{n\to\infty}\dfrac{\log L_n}{\varphi(n)}\leq\liminf_{i\to\infty}\dfrac{\log L_{l_i}}{\varphi(l_i)}=1.$$

Let
\begin{equation*}
Z:=\liminf\limits_{n\to \infty}\dfrac{\log \left(L_n^{t_0}L_{n+1}^{t_1}\cdots L_{n+m}^{t_{m+1}}\right)}{\varphi(n)}.
\end{equation*}
We claim that
\begin{equation*}
t_k\leq Z<\infty,
\end{equation*}
where $k=\min\{0\leq i\leq m:\ t_{i}\neq0\}$.
In fact, by the first part of \eqref{L j+1-L_j} and \eqref{L j+1-L_j-2}, we can check that
\begin{equation*}
Z\geq \liminf\limits_{n\to\infty}\dfrac{\log L_{n+k}^{t_k}}{\varphi(n)}\geq\liminf\limits_{n\to\infty}\dfrac{\log L_n^{t_k}}{\varphi(n)}= t_k.
\end{equation*}
On the other hand,
\begin{equation*}
\begin{aligned}
\liminf_{n\to\infty}\dfrac{\log \left(L_n^{t_0}L_{n+1}^{t_1}\cdots L_{n+m}^{t_{m+1}}\right)}{\varphi(n)}&\overset{\eqref{L j+1-L_j}}\leq \liminf\limits_{n\to \infty}\dfrac{\left(t_0+t_1(B+\epsilon)+\cdots+t_{m+1}{(B+\epsilon)}^m\right)\log L_n}{\varphi(n)}\\
&\overset{\eqref{L j+1-L_j-2}}= t_0+t_1(B+\epsilon)+\cdots+t_{m+1}{(B+\epsilon)}^m<\infty.
\end{aligned}
\end{equation*}
\textbf{Step \uppercase\expandafter{\romannumeral2}.}  We use the sequence $\{L_j\}_{j\geq1}$ and $Z$ to construct a subset of $\underline{E}(\{t_i\}_{i=0}^m,\varphi).$  Define

\begin{equation*}
E(\{L_n\}_{n\geq1})=\left\{x\in[0,1):\lfloor L^{\frac 1 Z}_n\rfloor\leq a_n(x)<2\lfloor L^{\frac 1 Z}_n\rfloor,\ \text{for any}\ n\geq1\right\}.
\end{equation*}
Then
\begin{equation*}
\begin{aligned}
\liminf\limits_{n\to\infty}\dfrac{\log \left(a^{t_0}_na^{t_1}_{n+1}\cdots a^{t_{m+1}}_{n+m}\right)}{\varphi(n)}&=
\liminf\limits_{n\to\infty}\dfrac{\log \left(L^{\frac{t_0}{Z}}_nL^{\frac{t_1}{Z}}_{n+1}\cdots L^{\frac{t_{m+1}}{Z}}_{n+m}\right)}{\varphi(n)}\\
&=\frac{1}{Z}\liminf\limits_{n\to\infty}\dfrac{\log \left(L^{t_0}_nL^{t_1}_{n+1}\cdots L^{t_{m+1}}_{n+m}\right)}{\varphi(n)}=1.
\end{aligned}
\end{equation*}
Hence $E(\{L_n\}_{n\geq1})\subset\underline{E}(\{t_i\}_{i=0}^m,\varphi).$
Since $\varphi(n)/n\to \infty$ as $n\to \infty$, by the second part of \eqref{L j+1-L_j}, we see that
\begin{equation*}
\lim\limits_{n\to\infty}\dfrac{\log (L_1L_2\cdots L_n)}{n}=\infty.
\end{equation*}
By Lemma \ref{anyset}, we obtain

\begin{equation*}
\hdim \underline{E}(\{t_i\}_{i=0}^m,\varphi)\geq \hdim E(\{L_n\}_{n\geq1})=\dfrac{1}{2+\limsup\limits_{n\to\infty}\frac{\log L_{n+1}}{\log (L_1L_2\cdots L_n)}}\overset{\eqref{for limsup set}}\geq \frac{1}{B+1+\epsilon}.
\end{equation*}
Taking $\epsilon\to0$, we conclude
$$\hdim \underline{E}(\{t_i\}_{i=0}^m,\varphi)\geq \frac{1}{B+1}.$$
\qed

\section{ Proof of Theorem \ref{forinfinitelymanyn}}

In this section, we give a proof of Theorem \ref{forinfinitelymanyn}. We adopt the strategies in \cite{LiRa_16}.
The proof of the theorem splits into two parts: finding the upper bound and the lower bound separately.

\textbf{Upper bound:}
For $x\in\overline{E}(\{t_i\}_{i=0}^m,\varphi)$, for any $\epsilon>0$, there exist infinitely many $n$ such that
$$a^{t_0}_n(x)a^{t_1}_{n+1}(x)\cdots a^{t_m}_{n+m}\geq e^{(1-\epsilon)\varphi(n)}.$$

Then by \cite[Theorem 1.5]{BaHuKlWa_22},

\begin{equation*}
\hdim \overline{E}_{m}(\{t_i\}_{i=0}^m,\varphi)\leq\frac{1}{1+b}.
\end{equation*}

\textbf{Lower bound:} We construct a suitable Cantor subset of $\overline{E}(\{t_i\}_{i=0}^m,\varphi)$ in two steps.

\textbf{Step \uppercase\expandafter{\romannumeral1}.} We will construct a sequence $\{c_n\}_{n\geq1}$ of positive real numbers such that
$$\limsup\limits_{n\to\infty}\dfrac{\log\left(c^{t_0}_nc^{t_1}_{n+1}\cdots c^{t_m}_{n+m}\right)}{\varphi(n)}=1,$$
and
$$\limsup\limits_{n\to\infty}\dfrac{\log c_{n+1}}{\log\left(c_1c_2\cdots c_n\right)}\leq b+\epsilon-1.$$

For all $n\in \mathbb{N}$, let $\Phi(n)=\min\limits_{k\geq n} \varphi(k).$ Since $\varphi(n)\to \infty$, as $n\to \infty$, $\Phi(n)$ is well defined. Thus, $\Phi(n)\leq \varphi(n),$ $\Phi(n)\leq \Phi(n+1)$ for all $n\in \mathbb{N}$. We claim that
\begin{equation*}
\Phi(n)= \varphi(n), \ \text{infinitely many}\  n\in \mathbb{N}.
\end{equation*}
If not, there exists $N\in\mathbb{N}$ such that for any $n\geq N$, $\Phi(n)< \varphi(n).$ Then for $n\geq N$, $\Phi(n)< \min\limits_{k\geq n}\varphi(k),$ which contradicts to the definition of $\Phi(n)$.

We define a sequence $\{c_n\}_{n\geq1}$ as follows:
\begin{equation}\label{sequencecn}
\begin{aligned}
&c_1=c_2=\cdots=c_m=1,\ c^{t_m}_{m+1}=e^{\Phi(1)},\\
&c^{t_m}_{n+m}=\min\left\{\dfrac{e^{\Phi(n)}}{c^{t_0}_{n}c^{t_1}_{n+1}\cdots c^{t_{m-1}}_{n+m-1}},{(c_1c_2\cdots c_{n+m-1})}^{t_m(b+\epsilon-1)}\right\},\  \text{for}\  n\geq 2.
\end{aligned}
\end{equation}
Since $(t_0,t_1,\cdots,t_m)\in \mathbb{R}_+^{m+1}$ with $0<t_0\leq t_1\leq \cdots \leq t_m$ and $\Phi$ is nondecreasing, we have  $c_n\geq 1$ for all $n\geq1$. Thus
\begin{equation}\label{limsupb-1}
\limsup\limits_{n\to\infty}\frac{\log c_{n+1}}{\log (c_1c_2\cdots c_n)}\leq \limsup\limits_{n\to\infty}\frac{\log {(c_1c_2\cdots c_n)}^{b+\epsilon-1}}{\log (c_1c_2\cdots c_n)}=b+\epsilon-1.
\end{equation}
We also claim that
\begin{equation}\label{sequencec-n}
c^{t_m}_{n+m}=\frac{e^{\varphi(n)}}{c^{t_0}_{n}c^{t_1}_{n+1}\cdots c^{t_{m-1}}_{n+m-1}} \ \text{for infinitely many}\ n.
\end{equation}
In order to prove \eqref{sequencec-n}, we first show that
\begin{equation} \label{sequencec-n-1}
c^{t_m}_{n+m}=\frac{e^{\Phi(n)}}{c^{t_0}_{n}c^{t_1}_{n+1}\cdots c^{t_{m-1}}_{n+m-1}} \ \text{for infinitely many}\ n.
\end{equation}
If not, there exists $N\in\mathbb{N}$ such that for any $n\geq N$,

\begin{equation}\label{sequence-2}
\begin{cases}
&c_{n+m}={\left(c_1c_2\cdots c_{n+m-1}\right)}^{b+\epsilon-1}\\
&\frac{e^{\Phi(n)}}{c^{t_0}_{n}c^{t_1}_{n+1}\cdots c^{t_{m-1}}_{n+m-1}} >{\left(c_1c_2\cdots c_{n+m-1}\right)}^{t_m(b+\epsilon-1)}.
\end{cases}
\end{equation}
Then
\begin{equation}\label{sequence-3}
\begin{cases}
&c_{n+m}=c^{b+\epsilon}_{n+m-1}\\
&e^{\Phi(n)} >{\left(c_1c_2\cdots c_{n+m-1}\right)}^{b+\epsilon-1}\dot c^{t_0}_{n}c^{t_1}_{n+1}\cdots c^{t_{m-1}}_{n+m-1}>{\left(c_1c_2\cdots c_{n+m-1}\right)}^{t_m(b+\epsilon-1)}.
\end{cases}
\end{equation}
Therefore, by \eqref{sequence-2} and \eqref{sequence-3}
\begin{equation}\label{sequence-4}
\begin{aligned}
\prod_{k=1}^nc_k=&\left(\prod_{k=1}^{N+m-1}c_k\right)\cdot c_{N+m}\cdot c_{N+m+1}\cdots c_n\\
=&\left(\prod_{k=1}^{N+m-1}c_k\right)\cdot {\left(\prod_{k=1}^{N+m-1}c_k\right)}^{b+\epsilon-1}\cdot c_{N+m+1}\cdots c_n\\
=&\left(\prod_{k=1}^{N+m-1}c_k\right)\cdot {\left(\prod_{k=1}^{N+m-1}c_k\right)}^{b+\epsilon-1}\cdot{\left(\prod_{k=1}^{N+m-1}c_k\right)}^{(b+\epsilon-1)(b+\epsilon)}\\
\cdots&{\left(\prod_{k=1}^{N+m-1}c_k\right)}^{(b+\epsilon-1){(b+\epsilon)}^{n-N-m}}\\
=&{\left(\prod_{k=1}^{N+m-1}c_k\right)}^{{(b+\epsilon)}^{n-N-m+1}}.
\end{aligned}
\end{equation}
Combining \eqref{sequence-3} with \eqref{sequence-4}, we obtain
\begin{equation*}
\begin{aligned}
\liminf\limits_{n\to\infty}\frac{\Phi(n+1)}{n+1}
&>\liminf\limits_{n\to\infty}\dfrac{{\log \log \left(c_1c_2
\cdots c_{n+m}\right)}^{t_m(b+\epsilon-1)}}{n+1}\\
&=\liminf\limits_{n\to\infty}\dfrac{\log \log {\left(\prod\limits_{k=1}^{N+m-1}c_k\right)}^{t_m(b+\epsilon-1){(b+\epsilon)}^{n-N+1}}}{n+1}=\log(b+\epsilon).
\end{aligned}
\end{equation*}
Then
\begin{equation*}
\liminf\limits_{n\to \infty}\frac{\varphi(n+1)}{n+1}\geq \liminf\limits_{n\to \infty}\frac{\Phi(n+1)}{n+1}>\log (b+\epsilon)>\log b,
\end{equation*}
which contradicts to $\log b=\liminf\limits_{n\to \infty}\frac{\varphi(n)}{n}.$

Now we begin to prove \eqref{sequencec-n}.
If  the equality \eqref{sequencec-n-1} holds for some $n$ such that $\Phi(n)\neq\varphi(n)$, then $\Phi(n)=\Phi(n+1)$, and the equality \eqref{sequencec-n-1} holds for $n+1$, since
\begin{equation*}
\begin{aligned}
&\frac{e^{\Phi(n+1)}}{c^{t_0}_{n+1}c^{t_1}_{n+2}\cdots c^{t_{m-1}}_{n+m}}=\frac{e^{\Phi(n)}}{c^{t_0}_{n+1}c^{t_1}_{n+2}\cdots c^{t_{m-1}}_{n+m}}=c^{t_0}_{n}c^{t_1-t_0}_{n+1}\cdots c^{t_m-t_{m-1}}_{n+m}\\
\leq&{ \left(c_1c_2\cdots c_{n-1}\right)}^{t_0(b+\epsilon-1)} {\left(c_1c_2\cdots c_{n}\right)}^{(t_1-t_0)(b+\epsilon-1)}\\
&\cdots {\left(c_1c_2\cdots c_{n+m-1}\right)}^{(t_m-t_{m-1})(b+\epsilon-1)}\\
=&{(c_1c_2\cdots c_{n-1})}^{t_m(b+\epsilon-1)}c_n^{(t_m-t_0)(b+\epsilon-1)}\cdots c_{n+m-1}^{(t_m-t_{m-1})(b+\epsilon-1)}\\
<&{(c_1c_2\cdots c_{n+m-1})}^{t_m(b+\epsilon-1)}.
\end{aligned}
\end{equation*}
By the fact that $\Phi(n)=\varphi(n)$ for infinitely many $n\in\mathbb{N}$, we can repeat this argument until we get to some $n+k$ such that $\Phi(n+k)=\varphi(n+k)$. Then the desired result is obtained.

Combining \eqref{sequencecn} with \eqref{sequencec-n}, we have
\begin{equation}\label{sequence-5}
\limsup\limits_{n\to\infty}\dfrac{\log\left(c^{t_0}_nc^{t_1}_{n+1}\cdots c^{t_m}_{n+m}\right)}{\varphi(n)}=1.
\end{equation}

\textbf{Step \uppercase\expandafter{\romannumeral2}.}  We use the sequence $\{c_n\}_{n\geq1}$ to construct a subset of $\overline{E}(\{t_i\}_{i=0}^m,\varphi).$

By $\varphi(n)/n\to \infty$ as $n\to \infty$, we choose an increasing sequence $\{n_k\}_{k=1}^{\infty}$ such that for each $k\geq 1$
\begin{equation*}
\frac{\varphi(n)}{n}\geq k^2,\ \text{when} \ n\geq n_k.
\end{equation*}
Let $\alpha_n=2$ if $1\leq n<n_1$ and
$$\alpha_n=k+1,\ \text{when}\ n_k\leq n<n_{k+1}.$$
For any $n\geq 1$, there exists $k(n)$ such that $n_{k(n)}\leq n+m<n_{k(n)+1}$. Then
\begin{equation}\label{sequence-alpha1}
\lim\limits_{n\to \infty}\dfrac{\log\left(\alpha^{t_0}_n\alpha^{t_1}_{n+1}\cdots\alpha^{t_m}_{n+m}\right)}{\varphi(n)}\leq \lim\limits_{n\to\infty}\dfrac{(t_0+\cdots+t_m)\log(k(n)+1)}{n{k(n)}^2}=0,
\end{equation}
and
\begin{equation}\label{sequence-alpha2}
\lim\limits_{n\to\infty}\dfrac{\log \alpha_{n+1}}{\log(\alpha_1\alpha_2\cdots\alpha_n)}\leq \lim\limits_{n\to\infty}\dfrac{\log(n+1)}{n\log2}=0.
\end{equation}
For any $n\geq1,$ take $s_n=c_n+\alpha_n$. Then  we have $s_n\to\infty$ as $n\to\infty$.

Define
\begin{equation*}
E(\{s_n\}_{n\geq1})=\left\{x\in[0,1):\lfloor s_n\rfloor\leq a_n(x)<2\lfloor s_n\rfloor,\ \text{for any}\ n\geq1\right\}.
\end{equation*}
Since $c_n\geq 1$ and $\alpha_n\geq2$ for all $n\geq1$, we can check that for any $n\geq1$
$$\log c_n\leq \log s_n\leq \log c_n +2\log \alpha_n.$$
Combining \eqref{sequence-5}, \eqref{sequence-alpha1}, \eqref{sequence-alpha2}, we get
\begin{equation*}
\limsup\limits_{n\to\infty}\dfrac{\log \left(s^{t_0}_ns^{t_1}_{n+1}\cdots s^{t_m}_{n+m}\right)}{\varphi(n)}=1.
\end{equation*}
So $E(\{s_n\}_{n\geq1})\subset\overline{E}(\{t_i\}_{i=0}^m,\varphi).$
Applying Lemma \ref{anyset}, we obtain
\begin{equation*}
\hdim \overline{E}(\{t_i\}_{i=0}^m,\varphi)\geq \hdim E(\{s_n\}_{n\geq1})=\dfrac{1}{2+\limsup\limits_{n\to\infty}\frac{\log s_{n+1}}{\log (s_1s_2\cdots s_n)}}\overset{\eqref{limsupb-1}}\geq \frac{1}{b+1+\epsilon}.
\end{equation*}
Therefore,
$$\hdim\overline{E}(\{t_i\}_{i=0}^m,\varphi)\geq \frac{1}{b+1}.$$

%\qed

\section*{Acknowledgements} A. Bakhtawar is supported by the Australian Research Council Discovery Project (ARC Grant
DP180100201) and J. Feng is supported by the National Natural Science Foundation of China (NSFC Grant No. 11901204). J. Feng would like to thank China Scholarship Council financial support (No. 202106160053). The authors are grateful to  Professor Lingmin Liao for helpful discussions.

\end{document}